\documentclass[12pt]{amsart}

\usepackage{amssymb}

\usepackage{enumitem}

\usepackage{graphicx}

\makeatletter
\@namedef{subjclassname@2020}{%
  \textup{2020} Mathematics Subject Classification}
\makeatother

\usepackage[T1]{fontenc}

\newtheorem{theorem}{Theorem}[section]
\newtheorem{corollary}[theorem]{Corollary}
\newtheorem{lemma}[theorem]{Lemma}
\newtheorem{proposition}[theorem]{Proposition}

\theoremstyle{definition}

\numberwithin{equation}{section}

\newcommand{\Mat}{\mathrm{Mat}_2(\mathbb{F})}
\begin{document}

\baselineskip=17pt

%%%%%%%%%%%%%%%%

\title[On $2$-dimensional invariant subspaces of matrices]{On $2$-dimensional invariant subspaces of matrices}

\author[O. Al-Raisi]{O. Al-Raisi}
\address{Department of  Mathematics,  College of Science, Sultan Qaboos University, Muscat, Oman}
\email{omartalibmiran@gmail.com}

\author[M. Shahryari]{M. Shahryari}
\address{Department of  Mathematics,  College of Science, Sultan Qaboos University, Muscat, Oman}
\email{m.ghalehlar@squ.edu.om}

\date{}

\begin{abstract}
We introduce a unified method for study of $2$-dimensional invariant subspaces of matrices and their corresponding {\em super-eigenvalues}. As a novel application to non-commutative algebra, we present a connection between the eigenvalues of  matrices with entries in the ring $\Mat$ and $2$-dimensional invariant subspaces of matrices with entries in the field $\mathbb{F}$.
\end{abstract}



\maketitle

\section{Introduction}
Existence of common $2$-dimensional invariant subspaces for various pairs of matrices has been studied in details; for example, in \cite{Ern}, the author investigates the common $2$-dimensional invariant subspaces of two rotations. Also in \cite{Ikram1} and \cite{Ikram2}, the problem of existence of a common $2$-dimensional (as well as higher dimensional) invariant subspaces of two matrices is taken into consideration. In this note, we introduce a unified method of the study of this concept as well as the notion of $2$-dimensional eigenvalues ({\em super-eigenvalues}) of matrices. A $4$-variable polynomial will be introduced which plays partially the role of the characteristic polynomial in the case of $2$-dimensional invariant subspaces. Then a necessary and sufficient condition for a $2\times 2$ matrix to be a super-eigenvalue of a matrix $A$ will be given in the case of files with zero or odd characteristics. We will see that there is a deep relation between these results and the problem of determining eigenvalues of linear maps on free modules over the non-commutative ring of $2\times 2$ matrices with entries in a field.

This connection can be described as follows:for a unital commutative ring $R$, the study of eigenvalues of matrices with entries in $R$ reduces to the problem of finding the roots of the corresponding characteristic polynomial. The same cannot be said in the study of eigenvalues of matrices with entries in a non-commutative ring. Indeed, there is no satisfactory definition of the characteristic polynomial of a matrix with entries in a non-commutative ring due to the lack of a proper definition of determinant in this case. We establish a connection between the problem of finding eigenvalues and eigenvectors of matrices over the ring $R=\Mat$, of all $2\times 2$ matrices over a field $\mathbb{F}$, and the problem of finding $2$-dimensional invariant subspaces of matrices with entries in $\mathbb{F}$, thereby gaining the ability to solve the former problem by solving the latter.

\section{Free modules over $R$}
Consider the $\mathbb{F}$-space $\mathcal{U}=\mathbb{F}^{2n}=\mathbb{F}^n\oplus \mathbb{F}^n$ and define an action of $R=\Mat$ on $\mathcal{U}$ by
$$
\left[\begin{array}{cc}
         p&q\\
         r&s
         \end{array}\right]\cdot \left[ \begin{array}{c} \mathbf{u}\\ \mathbf{v}\end{array}\right]=\left[ \begin{array}{c} p\mathbf{u}+q\mathbf{v}\\ r\mathbf{u}+s\mathbf{v}\end{array}\right].
$$

It can be easily verified that $\mathcal{U}$ is an $R$-module. Note that the ring $R$ has an IBN (invariant basis number), so the concept of the rank is well-defined for free $R$-modules.

\begin{lemma}
Suppose that $n=2k$ is an even number. Then $\mathcal{U}\cong_RR^k$ is the free $R$-module of rank $k$.
\end{lemma}

\begin{proof}
Suppose $\{ e_1, e_2, \ldots, e_n\}$ is the standard basis of the vector space $\mathbb{F}^n$. For every integer $i$ with $1\leq i\leq k$, we define
$$
T_i=\left[ \begin{array}{c} e_i\\ e_{k+i}\end{array}\right]
$$
and we show that $\mathbf{B}=\{ T_1, T_2, \ldots, T_k\}$ is an $R$-basis of $\mathcal{U}$. Suppose $\left[ \begin{array}{c} \mathbf{u}\\ \mathbf{v}\end{array}\right]\in \mathcal{U}$. We have
$$
\mathbf{u}=\sum_{i=1}^k(p_ie_i+q_ie_{k+i}),\ \mathbf{v}=\sum_{i=1}^k(r_ie_i+s_ie_{k+i}),
$$
for some scalars $p_i, q_i, r_i, s_i\in \mathbb{F}$. This means that
$$
\left[ \begin{array}{c} \mathbf{u}\\ \mathbf{v}\end{array}\right]=\sum_{i=1}^k\left[\begin{array}{cc}
         p_i&q_i\\
         r_i&s_i
         \end{array}\right]\cdot\left[ \begin{array}{c} e_i\\ e_{k+i}\end{array}\right]=\sum_{i=1}^k\left[\begin{array}{cc}
         p_i&q_i\\
         r_i&s_i
         \end{array}\right]\cdot T_i
$$
As a result, $\mathbf{B}$ generates $\mathcal{U}$. To show that $\mathbf{B}$ is linearly independent, let
$$
\sum_{i=1}^k\left[\begin{array}{cc}
         p_i&q_i\\
         r_i&s_i
         \end{array}\right]\cdot T_i=\mathbf{0}.
$$
Then, we have
$$
\sum_{i=1}^k(p_ie_i+q_ie_{k+i})=\mathbf{0},\ \sum_{i=1}^k(r_ie_i+s_ie_{k+i})=\mathbf{0},
$$
which implies that all $p_i, q_i, r_i$, and $s_i$ are zero.
\end{proof}

For an arbitrary matrix $A\in \mathrm{Mat}_n(\mathbb{F})$, define a map $f_A:\mathcal{U}\to \mathcal{U}$ by
$$
f_A(\left[ \begin{array}{c} \mathbf{u}\\ \mathbf{v}\end{array}\right])=\left[ \begin{array}{c} A\mathbf{u}\\ A\mathbf{v}\end{array}\right].
$$

\begin{lemma}
The map $f_A$ is an $R$-homomorphism. If $n=2k$ is an even number, then every $R$-homomorphism $f:\mathcal{U}\to \mathcal{U}$ has the form $f=f_A$, for some $A\in \mathrm{Mat}_n(\mathbb{F})$.
\end{lemma}

\begin{proof}
That $f_A$ is an $R$-homomorphism can be verified directly. Now, let $f:\mathcal{U}\to \mathcal{U}$ be an arbitrary $R$-homomorphism. For the sake of simplicity, we consider the case $n=4$. To this end, let $e_1, e_2, e_3, e_4$ be the standard basis of $\mathbb{F}^4$. As we just saw, an $R$-basis of $\mathcal{U}$ is $\mathbf{B}=\{ T_1, T_2\}$, where
$$
T_1=\left[ \begin{array}{c} e_1\\ e_3\end{array}\right],\ T_2=\left[ \begin{array}{c} e_2\\ e_4\end{array}\right].
$$
Let $\mathbf{X}$ be the set consisting of the vectors
\begin{small}
$$
E_1=\left[ \begin{array}{c} e_1\\ \mathbf{0}\end{array}\right], E_2=\left[ \begin{array}{c} e_2\\ \mathbf{0}\end{array}\right], E_3=\left[ \begin{array}{c} e_3\\ \mathbf{0}\end{array}\right], E_4=\left[ \begin{array}{c} e_4\\ \mathbf{0}\end{array}\right],
$$
\end{small}
and
\begin{small}
$$
F_1=\left[ \begin{array}{c} \mathbf{0}\\ e_1\end{array}\right], F_2=\left[ \begin{array}{c} \mathbf{0}\\ e_2\end{array}\right], F_3=\left[ \begin{array}{c} \mathbf{0}\\ e_3\end{array}\right], F_4=\left[ \begin{array}{c} \mathbf{0}\\ e_4\end{array}\right].
$$
\end{small}
Then $\mathbf{X}$ is an $\mathbb{F}$-basis of $\mathcal{U}$. Furthermore, we have the following obvious relations
$$
T_1=E_1+F_3, \  T_2=E_2+F_4,
$$
as well as the relations
\begin{small}
$$
E_1=\left[\begin{array}{cc}
         1&0\\
         0&0
         \end{array}\right]\cdot T_1, \ E_2=\left[\begin{array}{cc}
         1&0\\
         0&0
         \end{array}\right]\cdot T_2,\ E_3=\left[\begin{array}{cc}
         0&1\\
         0&0
         \end{array}\right]\cdot T_1,\ E_4=\left[\begin{array}{cc}
         0&1\\
         0&0
         \end{array}\right]\cdot T_2,
$$
\end{small}
and
\begin{small}
$$
F_1=\left[\begin{array}{cc}
         0&0\\
         1&0
         \end{array}\right]\cdot T_1, \ F_2=\left[\begin{array}{cc}
         0&0\\
         1&0
         \end{array}\right]\cdot T_2,\ F_3=\left[\begin{array}{cc}
         0&0\\
         0&1
         \end{array}\right]\cdot T_1,\ F_4=\left[\begin{array}{cc}
         0&0\\
         0&1
         \end{array}\right]\cdot T_2.
$$
\end{small}
Now, suppose $[f]_{\mathbf{B}}$, the matrix representation of $f$ with respect to the $R$-basis $\mathbf{B}$, is
$$
\left[\begin{array}{cc}
         \left[\begin{array}{cc}
         p_{1,1}&q_{1,1}\\
         r_{1,1}&s_{1,1}
         \end{array}\right]&\left[\begin{array}{cc}
         p_{1,2}&q_{1,2}\\
         r_{1,2}&s_{1,2}
         \end{array}\right]\\
         \ &\ \\
         \left[\begin{array}{cc}
         p_{2, 1}&q_{2, 1}\\
         r_{2, 1}&s_{2, 1}
         \end{array}\right]&\left[\begin{array}{cc}
         p_{2, 2}&q_{2, 2}\\
         r_{2, 2}&s_{2, 2}
         \end{array}\right]
         \end{array}\right].
$$
Then, we have
\begin{eqnarray*}
f(E_1)&=&\left[\begin{array}{cc}
         1&0\\
         0&0
         \end{array}\right]\cdot f(T_1)\\
      &=&  \left[\begin{array}{cc}
         1&0\\
         0&0
         \end{array}\right]\cdot ( \left[\begin{array}{cc}
         p_{1,1}&q_{1,1}\\
         r_{1,1}&s_{1,1}
         \end{array}\right]\cdot T_1+\left[\begin{array}{cc}
         p_{2, 1}&q_{2, 1}\\
         r_{2, 1}&s_{2, 1}
         \end{array}\right]\cdot T_2)\\
      &=&   \left[\begin{array}{cc}
         1&0\\
         0&0
         \end{array}\right]\cdot(  \left[\begin{array}{c}
         p_{1,1}e_1+q_{1,1}e_3\\
         r_{1,1}e_1+s_{1,1}e_3
         \end{array}\right]+  \left[\begin{array}{c}
         p_{2, 1}e_2+q_{2, 1}e_4\\
         r_{2, 1}e_2+s_{2, 1}e_4
         \end{array}\right]) \\
       &=& \left[\begin{array}{c}
         p_{1,1}e_1+p_{2, 1}e_2+q_{1,1}e_3+q_{2, 1}e_4\\
                  \mathbf{0}
         \end{array}\right]\\
       &=&p_{1,1}E_1+p_{2, 1}E_2+q_{1,1}E_3+q_{2, 1}E_4.
\end{eqnarray*}
Consequently, the first column of $[f]_{\mathbf{X}}$ is $[p_{1,1}\ p_{2, 1}\ q_{1,1}\ q_{2, 1}\ 0\ 0\ 0\ 0]^t$. Similarly, the second column is
$[p_{1,2}\ p_{2, 2}\ q_{1,2}\ q_{2, 2}\ 0\ 0\ 0\ 0]^t$. As a result, the matrix representation of $f$ as an $\mathbb{F}$-map has the form
$$
\left[ \begin{array}{cccccccc}
p_{1,1}&p_{1,2}&r_{1,1}&r_{1,2}&0&0&0&0\\
p_{2, 1}&p_{2, 2}&r_{2, 1}&r_{2, 2}&0&0&0&0\\
q_{1,1}&q_{1,2}&s_{1,1}&s_{1,2}&0&0&0&0\\
q_{2, 1}&q_{2, 2}&s_{2, 1}&s_{2, 2}&0&0&0&0\\
0&0&0&0&p_{1,1}&p_{1,2}&r_{1,1}&r_{1,2}\\
0&0&0&0&p_{2, 1}&p_{2, 2}&r_{2, 1}&r_{2, 2}\\
0&0&0&0&q_{1,1}&q_{1,2}&s_{1,1}&s_{1,2}\\
0&0&0&0&q_{2, 1}&q_{2, 2}&s_{2, 1}&s_{2, 2} \end{array}\right].
$$
Hence, if we consider the matrix
$$
A=\left[ \begin{array}{cccc}
p_{1,1}&p_{1,2}&r_{1,1}&r_{1,2}\\
p_{2, 1}&p_{2, 2}&r_{2, 1}&r_{2, 2}\\
q_{1,1}&q_{1,2}&s_{1,1}&s_{1,2}\\
q_{2, 1}&q_{2, 2}&s_{2, 1}&s_{2, 2} \end{array}\right],
$$
then $f=f_A$.
\end{proof}

It is easy to verify that, for a given matrix $A\in\mathrm{Mat}_n(\mathbb{F})$, we also have
$$
[f_A]_{\mathbf{X}}=\left[ \begin{array}{ccc}
\left[ \begin{array}{ll}a_{1,1}& a_{k+1, 1}\\
   a_{1, k+1}& a_{k+1, k+1}\end{array}\right]&\cdots&\left[ \begin{array}{ll}a_{1, k}& a_{k+1, k}\\
   a_{1, n}& a_{k+1, n}\end{array}\right]\\
   \vdots&\vdots  &\vdots\\
 \left[ \begin{array}{ll}a_{k, 1}& a_{n, 1}\\
   a_{k, k+1}& a_{n, k+1}\end{array}\right]&\cdots&\left[ \begin{array}{ll}a_{k, k}& a_{n, k}\\
   a_{k, n}& a_{n, n}\end{array}\right]  \end{array} \right].
$$
This matrix will henceforth be denoted by $\tilde{A}$. In the next section, the natural identification of the rings $\alpha:\mathrm{Mat}_k(R)\to \mathrm{Mat}_n(\mathbb{F})$ will be exploited. This identification removes the inner brackets in any $k\times k$ matrix
$$
M=\left[ \begin{array}{ccc}\left[\begin{array}{cc} a_{1, 1}^{(1, 1)}&a_{1, 2}^{(1,1)}\\
                               a_{2, 1}^{(1, 1)}&a_{2, 2}^{(1, 1)}\end{array}\right]&\cdots&\left[\begin{array}{cc} a_{1, 1}^{(1, k)}&a_{1, 2}^{(1,k)}\\
                               a_{2, 1}^{(1, k)}&a_{2, 2}^{(1, k)}\end{array}\right]\\
                               \vdots&\cdots&\vdots\\
                               \left[\begin{array}{cc} a_{1, 1}^{(k, 1)}&a_{1, 2}^{(k,1)}\\
                               a_{2, 1}^{(k, 1)}&a_{2, 2}^{(k, 1)}\end{array}\right]&\cdots&\left[\begin{array}{cc} a_{1, 1}^{(k, k)}&a_{1, 2}^{(k,k)}\\
                               a_{2, 1}^{(k, k)}&a_{2, 2}^{(k, k)}\end{array}\right]\end{array}
                               \right]
$$
and returns the $n\times n$ matrix
$$
\alpha(M)=\left[ \begin{array}{ccccc}
a_{1,1}^{(1, 1)}&a_{1,2}^{(1, 1)}&\cdots&a_{1,1}^{(1, k)}&a_{1,2}^{(1, k)}\\
a_{2,1}^{(1, 1)}&a_{2,2}^{(1, 1)}&\cdots&a_{2,1}^{(1, k)}&a_{2,2}^{(1, k)}\\
\vdots &\vdots &\cdots&\vdots&\vdots\\
a_{1,1}^{(k, 1)}&a_{1,2}^{(k, 1)}&\cdots&a_{1,1}^{(k, k)}&a_{1,2}^{(k, k)}\\
a_{2,1}^{(k, 1)}&a_{2,2}^{(k, 1)}&\cdots&a_{2,1}^{(k, k)}&a_{2,2}^{(k, k)}\end{array}\right].
$$
Consequently, for a matrix $A\in \mathrm{Mat}_n(\mathbb{F})$, we will denote the corresponding matrix $\alpha(\tilde{A})$ by $\hat{A}$.

\section{Super-eigenvalues}
For a matrix $A\in \mathrm{Mat}_n(\mathbb{F})$, the ordinary eigenvectors determine 1-dimensional invariant subspaces of $A$. We generalize this idea by defining a {\em super-eigenvector} of $A$ to be a $2$-dimensional invariant subspace of $A$. A super-eigenvector $W$ is termed {\em proper} if it does not contain any eigenvector of $A$. Let $W=\langle \mathbf{u}, \mathbf{v}\rangle$ be a super-eigenvector of $A$. Then
$$
A\mathbf{u}=p\mathbf{u}+q\mathbf{v}, \ A\mathbf{v}=r\mathbf{u}+s\mathbf{v}
$$
for some scalars $p, q, r$, and $s$. We call the matrix $\Lambda=\left[ \begin{array}{cc} p&q\\ r&s\end{array}\right]$ a {\em super-eigenvalue} of $A$ corresponding to $W$. It is easy to see that if $W=\langle \mathbf{u}^{\prime}, \mathbf{v}^{\prime}\rangle$ and
$$
A\mathbf{u}^{\prime}=p^{\prime}\mathbf{u}^{\prime}+q^{\prime}\mathbf{v}^{\prime}, \ A\mathbf{v}^{\prime}=r^{\prime}\mathbf{u}^{\prime}+s^{\prime}\mathbf{v}^{\prime},
$$
then the matrices
$$
\Lambda=\left[ \begin{array}{cc} p&q\\ r&s\end{array}\right], \Lambda^{\prime}=\left[ \begin{array}{cc} p^{\prime}&q^{\prime}\\ r^{\prime}&s^{\prime}\end{array}\right]
$$
are similar. Hence, every super-eigenvector of $A$ corresponds to a similarity class of $2\times 2$ matrices. A super-eigenvalue $\Lambda=\left[ \begin{array}{cc} p&q\\ r&s\end{array}\right]$ is called {\em proper} if and only if the corresponding super-eigenvector is proper. One can see that if $\Lambda$ is proper, then every super-eigenvector corresponding to $\Lambda$ is proper.

Now, let $n=2k$ be an even number. As we saw in the previous section, every $R$-homomorphism $f:\mathcal{U}\to \mathcal{U}$ has the form $f_A$, for some $A\in \mathrm{Mat}_n(\mathbb{F})$. In other words, every $n\times n$ matrix $A$ with entries in the field $\mathbb{F}$ corresponds to a unique $k\times k$ matrix $\tilde{A}$ with entries in $R$. Hence, we obtain the following at once.

\begin{proposition}
A matrix $\Lambda\in R$ is a super-eigenvalue of $A$ if and only if $\Lambda$ is an eigenvalue of $\tilde{A}$.
\end{proposition}

As a result, we have the following.

\begin{corollary}
A matrix $\Lambda=\left[ \begin{array}{cc} p&q\\ r&s\end{array}\right]$ is a super-eigenvalue of $A$ if and only if
$$
\det(\hat{A}-\Lambda\cdot I_k)=0,
$$
where by $\Lambda\cdot I_k$ we mean the matrix
$$
\left[ \begin{array}{ccccccc}
p&q&0&0&\cdots&0&0\\
r&s&0&0&\cdots&0&0\\
0&0&p&q&\cdots&0&0\\
0&0&r&s&\cdots&0&0\\
\vdots&\vdots&\vdots&\vdots&\cdots&\vdots&\vdots\\
0&0&0&0&\cdots&p&q\\
0&0&0&0&\cdots&r&s \end{array}\right].
$$
\end{corollary}

Note that the above statement is valid only in the case when $n$ is even. For the general case, we have the following.

\begin{proposition}
Let $\Lambda=\left[ \begin{array}{cc} p&q\\ r&s\end{array}\right]$ be a super-eigenvalue of $A$. Then we have
$$
\det((A-pI_n)(A-sI_n)-qrI_n)=0.
$$
\end{proposition}

\begin{proof}
Let $W=\langle \mathbf{u}, \mathbf{v}\rangle$ be a super-eigenvector corresponding to $\Lambda$ so that
$$
A\mathbf{u}=p\mathbf{u}+q\mathbf{v}, \ A\mathbf{v}=r\mathbf{u}+s\mathbf{v}.
$$
Hence $(A-pI_n)\mathbf{u}=q\mathbf{v}$, and therefore,
$$
(A-sI_n)(A-pI_n)\mathbf{u}=q(A-sI_n)\mathbf{u}=qr\mathbf{u}.
$$
So, we have
$$
((A-pI_n)(A-sI_n)-qrI_n)\mathbf{u}=\mathbf{0};
$$
since $\mathbf{u}\neq \mathbf{0}$, we obtain
$$
\det((A-pI_n)(A-sI_n)-qrI_n)=0.
$$
\end{proof}

Note that the converse of this proposition is not true: Let $p$ be an eigenvalue of $A$ and suppose that $s$ is not an eigenvalue of $A$. Then obviously $(p, 0, 0, s)$ satisfies the equation $\det((A-pI_n)(A-sI_n)-qrI_n)=0$ while the matrix $\Lambda=\left[ \begin{array}{cc} p&0\\ 0&s\end{array}\right]$ is not a super-eigenvalue of $A$.

In what follows, $p_{\Lambda}(T)$ denotes the characteristic polynomial of $\Lambda$; that is,
$$p_{\Lambda}(T)=(T-p)(T-s)-qr=T^2-\mathrm{tr}(\Lambda)T+\det(\Lambda).
$$

With this notation, the above proposition says that if $\Lambda$ is a super-eigenvalue of $A$ then $\det(p_{\Lambda}(A))=0$.

\begin{theorem}
Suppose that the characteristic of the field $\mathbb{F}$ is not $2$. A matrix $\Lambda=\left[ \begin{array}{cc} p&q\\ r&s\end{array}\right]$ is a proper super-eigenvalue of $A$ if and only if $\det(p_{\Lambda}(A))=0$ and $p_{\Lambda}(T)$ is irreducible.
\end{theorem}

\begin{proof}
First assume that $\Lambda$ is a proper super-eigenvalue of $A$. We already saw that $\det(p_{\Lambda}(A))=0$. Suppose $p_{\Lambda}(T)$ is reducible. Then $\Lambda$ is similar to a triangular matrix, say $\left[ \begin{array}{cc}a&b\\ 0&c\end{array}\right]$. Let $W=\langle \mathbf{u}, \mathbf{v}\rangle$ be a super-eigenvector corresponding to $\Lambda$. Then
$$
A\mathbf{u}=p\mathbf{u}+q\mathbf{v}, \ A\mathbf{v}=r\mathbf{u}+s\mathbf{v},
$$
however, there are two other vectors $\mathbf{u}^{\prime}, \mathbf{v}^{\prime}\in W$ such that
$$
A\mathbf{u}^{\prime}=a\mathbf{u}^{\prime}+b\mathbf{v}^{\prime}, \ A\mathbf{v}^{\prime}=c\mathbf{v}^{\prime}.
$$
As a result, $\mathbf{v}^{\prime}$ is an eigenvector of $A$; a contradiction.

Now, suppose that $\det(p_{\Lambda}(A))=0$ and $p_{\Lambda}(T)$ is irreducible. We prove that there is a proper super-eigenvector of $A$ corresponding to $\Lambda$. As $p_{\Lambda}(A)$ is singular, there exists a non-zero $\mathbf{u}$ such that $p_{\Lambda}(A)\mathbf{u}=\mathbf{0}$ which means
$$
(A-pI_n)(A-sI_n)\mathbf{u}=qr\mathbf{u}.
$$
Note that $qr\neq 0$, because otherwise $p_{\Lambda}(T)=(T-p)(T-s)$ which is reducible. Let
$$
\mathbf{v}=q^{-1}(A-pI_n)\mathbf{u}.
$$
Then $A\mathbf{u}=p\mathbf{u}+q\mathbf{v}$ and
\begin{eqnarray*}
A\mathbf{v}&=&q^{-1}A(A-pI_n)\mathbf{u}\\
           &=&q^{-1}(A-sI_n+sI_n)(A-pI_n)\mathbf{u}\\
           &=&q^{-1}(qr\mathbf{u})+q^{-1}sq\mathbf{v}\\
           &=&r\mathbf{u}+s\mathbf{v}.
\end{eqnarray*}
In order to prove that $W=\langle \mathbf{u}, \mathbf{v}\rangle$ is a corresponding super-eigenvector of $\Lambda$, we must show that $\mathbf{u}$ and $\mathbf{v}$ form a  linearly independent set. Suppose on the contrary that $\mathbf{v}=\lambda \mathbf{u}$, for some scalar $\Lambda$. Then $A\mathbf{u}=p\mathbf{u}+q\mathbf{v}=(p+q\lambda )\mathbf{u}$, and similarly, $A\mathbf{v}=(p+q\lambda )\mathbf{v}$. Hence,
$$
r\mathbf{u}+s\mathbf{v}=(p+q\lambda )\mathbf{v},
$$
which implies that
$$
r+s\lambda=(p+q\lambda)\lambda.
$$
As a result, $q\lambda^2+(p-s)\lambda-r=0$ and this means that $\lambda$ is a root of the quadratic equation $qT^2-(p-s)T-r=0$. As the characteristic of the field is not $2$, the discriminant
$$
(p-s)^2+4qr
$$
is a prefect square. But this is the discriminant of $p_{\Lambda}(T)$ which violates the assumption that $p_{\Lambda}(T)$ is irreducible. This shows that $\mathbf{u}$ and $\mathbf{v}$ form a linearly independent set, and consequently, $W$ is a super-eigenvector corresponding to $\Lambda$. It remains to show that $W$ is proper. To this end, let $\mathbf{v}^{\prime}\in W$ be an eigenvector of $A$ and let $\mathbf{u}^{\prime}\in W$ be any vector such that $W=\langle \mathbf{u}^{\prime}, \mathbf{v}^{\prime}\rangle$. Then,
$$
A\mathbf{u}^{\prime}=a\mathbf{u}^{\prime}+b\mathbf{v}^{\prime}, \ A\mathbf{v}^{\prime}=c\mathbf{v}^{\prime}
$$
for some scalars $a, b,$ and $c$. Hence the matrix $\Lambda^{\prime}=\left[ \begin{array}{cc} a&b\\ 0&c\end{array}\right]$ is also a corresponding super-eigenvalue of $W$. Consequently, $\Lambda$ and $\Lambda^{\prime}$ are similar. This is impossible as the characteristic polynomial of $\Lambda^{\prime}$ is reducible.
\end{proof}

We close this section by listing a few consequences of the above theorem. Again, we assume that the field $\mathbb{F}$ is not of characteristic $2$.

\begin{corollary}
A matrix $\Lambda$ is a proper super-eigenvalue of $A$ if and only if its characteristic polynomial $p_{\Lambda}(T)$ is an irreducible factor of the characteristic polynomial of $A$.
\end{corollary}

\begin{proof}
Let $\Lambda$ be a super-eigenvalue of $A$ corresponding to $W=\langle \mathbf{u}, \mathbf{v}\rangle$. Extending the set $\{ \mathbf{u}, \mathbf{v}\}$ to a basis of $\mathbf{F}^n$, we see that $A$ is similar to a matrix of the form
$$
\left[ \begin{array}{cc}
         \Lambda^t&B\\
         \mathbf{0}&C \end{array}\right],
$$
and hence $p_{\Lambda}(T)$ is a factor of $p_A(T)$, the characteristic polynomial of $A$. Conversely, suppose
$$
p_A(T)=(T^2+\lambda T+\mu)q(T),
$$
where the quadratic polynomial $T^2+\lambda T+\mu$ is irreducible. As $p_A(A)=\mathbf{0}$, the matrix  $A^2+\lambda A+\mu I_n$ is singular, and hence, there is a non-zero vector $\mathbf{u}$ such that $(A^2+\lambda A+\mu I_n)\mathbf{u}=\mathbf{0}$. Let $\mathbf{v}=A\mathbf{u}+\lambda \mathbf{u}$. Then we have
$$
A\mathbf{u}=-\lambda \mathbf{u}+\mathbf{v}, \ A\mathbf{v}=-\mu \mathbf{u}.
$$
We show that $W=\langle \mathbf{u}, \mathbf{v}\rangle$ is a  proper super-eigenvector of $A$ corresponding to the matrix $\Lambda=\left[ \begin{array}{rr} -\lambda& 1\\ -\mu& 0\end{array}\right]$. Note that $\dim(W)=2$, otherwise we must have $\mathbf{v}=\alpha \mathbf{u}$, for some scalar $\alpha$. But, then $A\mathbf{u}=(\alpha-\lambda)\mathbf{u}$ which implies that
$$
(\alpha-\lambda)^2+\lambda(\alpha-\lambda)+\mu=0
$$
which means that the polynomial $T^2+\lambda T+\mu$ is reducible; a contradiction. Now, $p_{\Lambda}(T)=T^2+\lambda T+\mu$ which is irreducible, so, using the previous theorem, $\Lambda$ is a proper super-eigenvalue of $A$.
\end{proof}

Consequently, in order to determine the (proper) eigenvalues and $R$-endomorphism over a free $R$-module, we must first determine the corresponding matrix $A$. Then, using the above result, we can find all proper super-eigenvalues of $A$.

\begin{corollary}
Let $R=\Mat$ be the ring of all $2\times 2$ matrices over a filed $\mathbb{F}$ whose characteristic is not $2$. Let $f:R^k\to R^k$ be an arbitrary $R$-homomorphism. Suppose $A$ is the corresponding matrix of $f$ and consider the set $\{ p_1(T), \ldots, p_m(T)\}$ of all quadratic irreducible factors of the characteristic polynomial $p_A(T)$. If $p_i(T)=T^2+\lambda_i T+\mu_i$, for all $i$, then the matrices
$$
\Lambda_i=\left[ \begin{array}{rr} -\lambda_i& 1\\ -\mu_i& 0\end{array}\right], \ (1\leq i\leq m)
$$
are all pairwise non-similar proper eigenvalues of $f$.
\end{corollary}

The number of pairwise non-similar proper eigenvalues of an $R$ homomorphism $f:R^k\to R^k$ is at most $k$. This can be proved as follows. Let $A$ be the corresponding matrix of $f$ and suppose that for each $1\leq i\leq m$, $W_i=\langle \mathbf{u}_i, \mathbf{v}_i\rangle$ is a proper super-eigenvector of $A$ with a corresponding super-eigenvalue $\Lambda_i$. Then we show that the sum $\sum_{i=1}^mW_i$ is direct. As we know, each characteristic polynomial $p_{\Lambda_i}(T)$ is an irreducible factor of $p_A(T)$. So, we have
$$
p_A(T)=p_1(T)^{\alpha_1}\cdots p_m(T)^{\alpha_m}\times \mathrm{other}\ \mathrm{factors},
$$
for some positive integers $\alpha_1, \ldots, \alpha_m$. Let $V_i=\ker(p_i(A)^{\alpha_i})$. By the primary decomposition theorem, the sum $\sum_{i=1}^mV_i$ is direct. On the other hand, a direct check shows that
$$
p_i(A)\mathbf{u}_i=p_i(A)\mathbf{v}_i=\mathbf{0},
$$
so $W_i\subseteq V_i$. This shows that the sum $\sum_{i=1}^mW_i$ is direct as well, and consequently, $A$ has only finitely many pairwise non-similar proper super-eigenvalue.

\section{An Example}
Suppose $\mathbb{F}=\mathbb{R}$ is the field of real numbers. Let $R=\mathrm{Mat}_2(\mathbb{R})$ and $\mathcal{U}=\mathbb{R}^4\oplus \mathbb{R}^4$. The map
$$
(x_1, x_2, x_3, x_4, y_1, y_2, y_3, y_4)\mapsto (-x_4, x_1, x_2, x_3, -y_4, y_1, y_2, y_3)
$$
is an $R$-homomorphism. The corresponding matrix is
$$
A=\left[ \begin{array}{cccc}0&0&0&-1\\
                            1&0&0&0\\
                            0&1&0&0\\
                            0&0&1&0
                            \end{array}\right].
$$
The characteristic polynomial of $A$ is $T^4+1$ which has no root, however, we show that $A$ has two (non-similar) proper super-eigenvalues
$$
\left[ \begin{array}{cc} 0&-1\\ 1&\sqrt{2}\end{array}\right], \ \left[ \begin{array}{cc} 0&-1\\ 1&-\sqrt{2}\end{array}\right].
$$
We have
$$
(A-pI_4)(A-sI_4)-qrI_4=\left[ \begin{array}{cccc}
                               d&0&-1&t\\
                               -t&d&0&-1\\
                               1&-t&d&0\\
                               0&1&-t&d
                                  \end{array}\right],
$$
where $d=ps-qr$ and $t=p+s$. The determinant of the above matrix is $t^4-4dt^2+(d^2+1)^2$. Hence, in order to find the super-eigenvalues, we must solve the equation $t^4-4dt^2+(d^2+1)^2=0$. This is a quadratic equation in terms of $t^2$ and its discriminant is
$$
16d^2-4(d^2+1)^2=-4(d^2-1)^2.
$$
We must have $-4(d^2-1)^2\geq 0$, so the only possibilities are $d=\pm 1$. The case $d=-1$ gives rise to the polynomial $t^4+4t^2+4=0$ which has no root. So, $d=1$, and consequently, $t=\pm\sqrt{2}$. This means that for any proper super-eigenvalue $\Lambda$, we have $\det(\Lambda)=1$ and $\mathrm{tr}(\Lambda)=\pm \sqrt{2}$. As such, the characteristic polynomial is $p_{\Lambda}(T)=T^2\pm \sqrt{2}T+1$ which is indeed irreducible. Therefore, $\Lambda$ has the same characteristic  and minimal polynomial as one of the matrices
$$
\left[ \begin{array}{cc} 0&-1\\ 1&\sqrt{2}\end{array}\right], \ \left[ \begin{array}{cc} 0&-1\\ 1&-\sqrt{2}\end{array}\right].
$$
This means that $\Lambda$ is similar to one of the above matrices. Now, we determine, the corresponding super-eigenvectors. First consider
$\Lambda=\left[ \begin{array}{cc} 0&-1\\ 1&\sqrt{2}\end{array}\right]$. We must solve the system
$$
\left[ \begin{array}{c}A\mathbf{u}\\ A\mathbf{v}\end{array}\right]=\Lambda\cdot \left[ \begin{array}{c}\mathbf{u}\\ \mathbf{v}\end{array}\right],
$$
where $\mathbf{u}$ and $\mathbf{v}$ are the vectors
$$
\mathbf{u}=\left[ \begin{array}{c}x_1\\ x_2\\ x_3\\ x_4 \end{array}\right], \ \mathbf{v}=\left[ \begin{array}{c}y_1\\ y_2\\ y_3\\ y_4 \end{array}\right].
$$
This gives rise to the linear equations
$$
\begin{array}{rrrrrrrrrr}
\ &\ &\ &x_4&-y_1&\ &\ &\ &=&0\\
-x_1&\ &\ &\ &\ &-y_2&\ &\ &=&0\\
\ &-x_2&\ &\ &\ &\ &-y_3&\ &=&0\\
\ &\ &-x_3&\ &\ &\ &\ &-y_4&=&0\\
x_1&\ &\ &\ &+\sqrt{2}y_1&\ &\ &+y_4&=&0\\
\ &x_2&\ &\ &-y_1&+\sqrt{2}y_2&\ &\ &=&0\\
\ &\ &x_3&\ &\ &-y_2&+\sqrt{2}y_3&\ &=&0\\
\ &\ &\ &x_4&\ &\ &-y_3&+\sqrt{2}y_4&=&0.
\end{array}
$$
Therefore, we have
$$
\mathbf{u}=\left[ \begin{array}{c}0\\ 1\\ \sqrt{2}\\ 1 \end{array}\right], \ \mathbf{v}=\left[ \begin{array}{c}1\\ 0\\ -1\\ -\sqrt{2} \end{array}\right],
$$
and consequently, the corresponding super-eigenvector is
$$
W=\langle \left[ \begin{array}{c}0\\ 1\\ \sqrt{2}\\ 1 \end{array}\right], \ \left[ \begin{array}{c}1\\ 0\\ -1\\ -\sqrt{2} \end{array}\right]\rangle.
$$
Similarly, the corresponding super-eigenvector of $\Lambda=\left[ \begin{array}{cc} 0&-1\\ 1&-\sqrt{2}\end{array}\right]$ is
$$
W=\langle \left[ \begin{array}{c}0\\ 1\\ -\sqrt{2}\\ 1 \end{array}\right], \ \left[ \begin{array}{c}1\\ 0\\ -1\\ \sqrt{2} \end{array}\right]\rangle.
$$
As a result, with respect to the basis
$$
\left[\begin{array}{c}0\\ 1\\ \sqrt{2}\\ 1 \end{array}\right], \ \left[ \begin{array}{c}1\\ 0\\ -1\\ -\sqrt{2} \end{array}\right], \ \left[ \begin{array}{c}0\\ 1\\ -\sqrt{2}\\ 1 \end{array}\right], \ \left[ \begin{array}{c}1\\ 0\\ -1\\ \sqrt{2} \end{array}\right]
$$
of $\mathbb{R}^4$, the matrix of the above map will have the block form
$$
\left[ \begin{array}{rrrr}
 0&-1&0&0\\
 1&\sqrt{2}&0&0\\
 0&0&0&-1\\
 0&0&1&-\sqrt{2}
 \end{array} \right].
$$

\end{document}